\documentclass[10pt]{amsart}
\usepackage[english]{babel}
\usepackage{amssymb}
\usepackage{amsmath}

\usepackage{lscape}

\newtheorem{dummy}{Dummy}

\newtheorem{lemma}[dummy]{Lemma}
\newtheorem{theorem}[dummy]{Theorem}
\newtheorem{proposition}[dummy]{Proposition}
\newtheorem{corollary}[dummy]{Corollary}

\theoremstyle{definition}

\newcommand{\ignore}[1]{}

\author{C. Brown}
\author{S. Pumpl\"un}

\email{christian\_jb@hotmail.co.uk; susanne.pumpluen@nottingham.ac.uk}
\address{School of Mathematical Sciences\\
University of Nottingham\\ University Park\\ Nottingham NG7 2RD\\
United Kingdom }



\begin{document}

\title[Irreducible skew polynomials over domains]
{Irreducible skew polynomials over domains}

\begin{abstract}
Let $S$ be a domain and $R=S[t;\sigma,\delta]$ a
skew polynomial ring, where $\sigma$ is an injective endomorphism of $S$ and $\delta$ a left $\sigma$-derivation.
 We give criteria for skew polynomials  $f\in R$ of degree less or equal to four to be irreducible.
 We apply them to low degree polynomials in quantized Weyl algebras and the quantum planes.
 We also consider $f(t)=t^m-a\in R$.
\end{abstract}

\maketitle

%*******************************************************************************************%
%
\section*{Introduction}
%
%*******************************************************************************************%

Let $R=S[t;\sigma,\delta]$, where $\sigma$ is an injective endomorphism  of $S$ and $\delta$ a left
$\sigma$-derivation.
A sequence of well-known papers by Lam, Leroy and others greatly contributed to our understanding
 of skew polynomials
over division rings and their factorization behaviour, starting with \cite{lam1988vandermonde}.
Some earlier results are contained in \cite{B, C, P66}. More recently,
 two general algorithms for computing
 the bound of a skew polynomial over a skew field were given in \cite{GLN}. As an application,
 a criterion for deciding whether a bounded skew polynomial is irreducible was developed.
 The computational method presented there heavily relies on finding the zero divisors in certain
 central simple algebras and is only
  applicable for certain set-ups, where the input data $S$, $\sigma$ and $\delta$ are effective and computable.
  It works for bounded skew polynomials.
  However, most of the results  on the irreducibility of skew polynomials in $R$ obtained so far assume that $S$ is a division algebra.
 Some first results on factoring certain skew polynomials of degree two in
 quantum planes and quantized Weyl algebras were collected in \cite{CP, HP, CookP}.

 In this note, $S$ is a domain, i.e., a unital ring without non-trivial zero divisors. We look at some skew polynomials $f$ of low degree in $R$
 and when they are irreducible. We briefly consider the skew polynomial  $f(t)=t^m-a\in R$ as well.
 This way we hope to give the reader a tool kit to be generalized and used also on higher degree polynomials, as well as a collection of examples of irreducible polynomials.

 The paper is organized as follows: In Section \ref{sec:irreduciblebenodomain}, we collect
  some necessary and sufficient criteria for skew polynomials in
 $S[t;\sigma,\delta]$  of degree less than five to be irreducible, adjusting well known results for skew polynomials over division rings, whose proofs carry over without problems.  The situation is easiest when
  $S$ is a right Ore domain. To see if $f\in R$ to be irreducible,
we can simply check if $f$ is irreducible in  $D[t;\sigma,\delta]$,
where $D$ is the right ring of fractions of $S$. We employ this approach for $f(t)=t^m-a\in S[t;\sigma,\delta]$.
 In Section \ref{subsec:irred},
 we systematically look at irreducibility criteria for polynomials of degree two and three
 in  $K[y][t;\sigma]$.
 Following \cite{BI, BII, BIII}, we define $A_h=K[y][t; \delta]$ with $\delta(r) = r'h$ for some $h(y)
\in K[y]$, where $r'$ is the usual derivation of $r$ with respect to $y$. For instance, $A_y$ is the universal enveloping algebra of the two-dimensional non-abelian Lie algebra. The algebra
$A_{y^2}$ is also known as the Jordan plane which appears in noncommutative algebraic geometry.
Irreducible polynomials of degree two and three as well as $f(t)=t^m-a$ in the quantum plane,
 and irreducible skew polynomials
of degree two in a quantized Weyl algebra and in $A_h$  are then considered in Section \ref{sec:Weyletc}.

In particular, a monic polynomial $f \in K[t]$ of degree two or three is irreducible in the quantum plane
$K[y][t;\sigma]$ if and only if it is irreducible in $K[t]$ (Corollary \ref{thm:irredquantum}), and
a degree two polynomial  $f \in K[t]$ is irreducible in the quantized Weyl algebra
$K[y][t;\sigma,\delta]$ if and only if it is irreducible in $K[t]$ (Corollary \ref{thm:irredWeyl}).

%%%%%%%%%%%%%%%%%%%%%%%%%%%%%%%%%%%%%%%%%%%%%%%%%%%%%%%%%%%%%%%%%%%%%%%%%%%%%%%%
%
% Irreducibility criteria
%
%%%%%%%%%%%%%%%%%%%%%%%%%%%%%%%%%%%%%%%%%%%%%%%%%%%%%%%%%%%%%%%%%%%%%%%%%%%%%%%%%

\section{Irreducibility criteria for polynomials of low degree in $S[t;\sigma,\delta]$ where $S$ is a domain}
\label{sec:irreduciblebenodomain}

Let $S$ be a domain, $\sigma$  an injective endomorphism of $S$ and $\delta$  a left $\sigma$-derivation of $S$,
i.e. an additive map such that $\delta(ab)=\sigma(a)\delta(b)+\delta(a)b$ for all $a,b\in S$.
The \emph{skew polynomial ring} $R=S[t;\sigma,\delta]$ is the
set of skew polynomials $g(t)=a_0+a_1t+\dots +a_nt^n$ with $a_i\in
S$, with term-wise addition, where the multiplication is defined
via $ta=\sigma(a)t+\delta(a)$ for all $a\in S$ \cite{O1}. That means,
$$at^nbt^m=\sum_{j=0}^n a(\Delta_{n,j}\,b)t^{m+j}$$ for all $a,b\in
S$, where the map $\Delta_{n,j}$ is defined recursively via
$$\Delta_{n,j}=\delta(S_{n-1,j})+\sigma (\Delta_{n-1,j-1}),$$ with
$\Delta_{0,0}=id_S$, $\Delta_{1,0}=\delta$, $\Delta_{1,1}=\sigma $.
Thus $\Delta_{n,j}$ is the sum of all polynomials in $\sigma$
and $\delta$ of degree $j$ in $\sigma$ and degree $n-j$ in $\delta$
\cite[p.~2]{J96}. If $\delta=0$, then $\Delta_{n,n}=\sigma^n$.
Define $S[t;\sigma]=S[t;\sigma,0]$.

 For $f(t)=a_0+a_1t+\dots +a_nt^n\in R$ with $a_n\not=0$ define ${\rm deg}(f)=n$ and ${\rm deg}(0)=-\infty$.
 An element $f\in R$ is \emph{irreducible} in $R$ if it is not a unit and  it has no proper factors, i.e if there do not exist $g,h\in R$ with
 ${\rm deg}(g),{\rm deg} (h)<{\rm deg}(f)$ such that $f=gh$.

   $f \in R$ is \emph{bounded} if there exists $0 \neq f^* \in
R$ such that $Rf^* = f^* R$ is the largest two-sided ideal of $R$
contained in $Rf$. The element $f^*$ is determined by $f$ up to
multiplication on the left by elements of $S^{\times}$.

\subsection{}
Since $S$ is a domain, we have $\mathrm{deg}(gh)
=\mathrm{deg}(g) + \mathrm{deg}(h)$ for all $g, h \in S[t;\sigma,\delta]$. This implies that some results hold for $f$, which were so far only shown for $S$ a division algebra. We start by collecting them here for the convenience of the reader.

\begin{theorem} \label{thm:irreducibility criteria when S is a ring without zero divisors}
(proved analogously as  \cite[Theorem 25, Theorem 31]{CP18})
 (i) $f(t) = t^2 - a_1 t - a_0 \in S[t;\sigma]$ is irreducible if and only if
$\sigma(b)b - a_1 b - a_0 \neq 0$
for all $b \in S$.
\\ (ii) Suppose $\sigma \in {\rm Aut}(S)$. Then $f(t) = t^3 - a_2 t^2 - a_1 t - a_0 \in S[t;\sigma]$ is irreducible if and only if
\begin{equation*}
\sigma^2 (b) \sigma(b) b - \sigma^2 (b)\sigma(b) a_2 - \sigma^2 (b)
\sigma(a_1) - \sigma^2 (a_0) \neq 0,
\end{equation*}
and
\begin{equation*}
\sigma^2 (b) \sigma(b) b - a_2 \sigma(b) b - a_1 b - a_0 \neq 0
\end{equation*}
for all $b \in S.$
\\ (iii) Suppose $\sigma \in {\rm Aut}(S)$, then $f(t) = t^3 - a \in S[t;\sigma]$ is irreducible if and only if $\sigma^2 (b) \sigma(b) b \neq a$ for all $b \in S$.
\\ (iv) Suppose $\sigma \in {\rm Aut}(S)$, then $f(t) = t^4 - a_3 t^3 - a_2 t^2 - a_1 t - a_0 \in S[t;\sigma]$ is irreducible if and only if
\begin{equation} \label{eqn:degree 4 irreducible 1}
\sigma^3 (b) \sigma^2 (b) \sigma(b) b + a_3 \sigma^2 (b) \sigma(b) b
+ a_2 \sigma(b) b + a_1 b + a_0 \neq 0,
\end{equation}
and
\begin{equation} \label{eqn:degree 4 irreducible 2}
\begin{split}
&\sigma^3 (b) \sigma^2 (b) \sigma (b) b + \sigma^3 (b) \sigma^2 (b)
\sigma (b) a_3 \\ &+ \sigma^3 (b) \sigma^2 (b) \sigma (a_2) +
\sigma^3 (b) \sigma^2 (a_1) + \sigma^3 (a_0) \neq 0,
\end{split}
\end{equation}
for all $b \in S$, and for every $ c , d \in S$, we have either
\begin{equation} \label{eqn:degree 4 irreducible 3}
\sigma^2 (c) \sigma(c)c + \sigma^2 (d)c + \sigma^2 (c) \sigma(d) +
a_3 (\sigma(d) + \sigma(c)c) + a_2 c + a_1 \neq 0,
\end{equation}
 or
\begin{equation} \label{eqn:degree 4 irreducible 4}
\sigma^2 (d)d + \sigma^2 (c) \sigma(c) d + a_3 \sigma(c)d + a_2 d +
a_0 \neq 0.
\end{equation}

\end{theorem}

That means in Theorem \ref{thm:irreducibility criteria when S is a ring without zero divisors} (iv),
the skew polynomial $f(t)$ is irreducible if and only if \eqref{eqn:degree 4
irreducible 1} and \eqref{eqn:degree 4 irreducible 2} and
(\eqref{eqn:degree 4 irreducible 3} or \eqref{eqn:degree 4
irreducible 4}) holds.

\begin{corollary} (proved analogously as \cite[Corollary 33]{CP18})
 Suppose $\sigma \in {\rm Aut}(S)$. Then
 $f(t) = t^4 - a \in S[t;\sigma]$ is reducible if and only if
$$\sigma^2 (c) \sigma(c)c + \sigma^2 (d)c + \sigma^2 (c) \sigma(d) =
0 \quad \text{and} \quad \sigma^2 (d)d + \sigma^2 (c) \sigma(c) d =a,$$ for some $c, d \in S$.
\end{corollary}

We now recursively define  two sequences of maps: $N_i : S \rightarrow S$,
and $M_i:S \rightarrow S$, $i \geq 0$.
The maps $N_i$ are given via
$$N_0(b) = 1, \ N_{i+1}(b) = \sigma(N_i(b))b + \delta(N_i(b)), $$ i.e.,
$ N_1(b) = b, \ N_2(b) = \sigma(b)b + \delta(b),\ldots$
For the definition of the maps $M_i$ we
assume that $\sigma \in {\rm Aut}(S)$. Define
$$ M_0(b) = 1, \ M_{i+1}(b) = b \sigma^{-1}(M_i(b)) - \delta(\sigma^{-1}(M_i(b))),$$
i.e., $M_1(b) = b$, \ $M_2(b) = b \sigma^{-1}(b) - \delta(\sigma^{-1}(b)), \ldots$
\cite{lam1988vandermonde}.

Let $f(t) = t^m - \sum_{i=0}^{m-1} a_i t^i \in S[t;\sigma,\delta]$. Then $(t-b) \vert_r f(t)$ is equivalent to
$N_m(b) - \sum_{i=0}^{m-1} a_i N_i (b) = 0$
which is proved analogously as \cite[Lemma 2.4]{lam1988vandermonde}.

If $\sigma  \in {\rm Aut}(S)$, we can also view $R=S[t;\sigma,\delta]$ as a right
polynomial ring and write $f(t) = t^m -\sum_{i=0}^{m-1} a_i t^i \in R$ in the form $f(t) = t^m -
\sum_{i=0}^{m-1} t^i a_i'$ for some uniquely determined $a_i' \in S$.
 The remainder after dividing $f(t)$ on the left by $t-b$ is then given by  $M_m(b) - \sum_{i=0}^{m-1}
M_i(b) a_i'$ which is proved analogously as \cite[Proposition 49]{CP18}.

\begin{theorem}  \label{prop:irreducibility criteria when S is a ring without zero divisors delta not 0}
(proved analogously as  \cite[Theorem 50]{CP18})
Let $\sigma\in {\rm Aut}(S)$.
\\ (i) $f(t) = t^2 - a_1 t - a_0 \in S[t;\sigma,\delta]$ is irreducible if and only if
 $\sigma(b)b + \delta(b) - a_1 b - a_0 \neq 0$  for all $b \in S$.
\\ (ii) Suppose  $f(t) = t^3 - a_2 t^2 - a_1 t - a_0 \in S[t;\sigma,\delta]$.
Write $f(t) = t^3 - t^2 a_2' - t a_1' - a_0'$ for some unique $a_0', a_1', a_2' \in S$. Then $f(t)$ is irreducible if and
only if
$$N_3(b) - \sum_{i=0}^{2}a_i N_i(b) \neq 0
\text{ and }
M_3(b) - \sum_{i=0}^{2} M_i(b) a_i' \neq 0$$
for all $b \in S$.
\end{theorem}

%%%%%%%%%%%%%%%%%%%%%%%%%%%%%%%%%%%%%%%%%%%%%%%%%%%%%%%%%%%%%%%%%%%%%%%%%%%%%%%%%%%%%%%%%%
\subsection{Right Ore domains} \label{subsec:2}
%%%%%%%%%%%%%%%%%%%%%%%%%%%%%%%%%%%%%%%%%%%%%%%%%%%%%%%%%%%%%%%%%%%%%%%%%%%%%%%%%%%%%%%

A domain $S$ is a \emph{right Ore domain} if $aS \cap bS\neq \{ 0 \}$ for all $0 \neq a, b \in S$.
The \emph{ring of right
fractions} of $S$ is a division ring $D$ containing $S$, such that every element of $D$ is of the form $rs^{-1}$ for some $r \in S$ and $0 \neq s \in S$.
Moreover, $\sigma$ and $\delta$ extend uniquely to $D$ via
\begin{equation} \label{eqn:extend sigma delta to right ring of fractions}
\sigma(rs^{-1}) = \sigma(r)\sigma(s)^{-1} \text{ and }
\delta(rs^{-1}) = \delta(r)s^{-1} - \sigma(rs^{-1}) \delta(s) s^{-1},
\end{equation}
for all $r \in S, \ 0 \neq s \in S$ by \cite[Lemma 1.3]{Goodearl}.
Note that any integral domain is a right Ore domain; its  ring of right fractions is equal to its quotient field.
In this subsection, we assume that $S$ is a right Ore domain with ring of right fractions $D$, $\sigma$ is
 an injective endomorphism of $S$ and $\delta$ a
$\sigma$-derivation of $S$. Let $C(D)$ denote the center of $D$.

 If $S$ is a right Ore domain,
we can take a skew polynomial $f\in S[t;\sigma,\delta]$ and check if it is irreducible in  $D[t;\sigma,\delta]$, in which case it
will be irreducible in $ S[t;\sigma,\delta]$ as well. We thus obtain the following results as elementary
corollaries of
the corresponding results  \cite[Theorems 39 and 51, Corollary 52]{CP18}:

\begin{theorem} \label{thm:Petit(19)}
Suppose $m$ is prime and $C(D) \cap \mathrm{Fix}(\sigma)$
contains a primitive $m$th root of unity.
\\ (i)  $f(t) = t^m - a \in S[t;\sigma]$ is irreducible if
$$a \neq \sigma^{m-1} (b) \cdots\sigma(b) b$$
 for all $b \in D$.
 \\ (ii) If $\mathrm{char}(D) \neq m$ then $f(t) = t^m - a \in
S[t;\sigma,\delta]$ is irreducible if $N_m(b) \neq a$ for
all $b \in D$.
 \\ (iii) Suppose $m=3$ and $\mathrm{char}(D) \neq 3$. Then $f(t) = t^3-
a \in S[t; \delta]$ is irreducible if $$N_3(b)=b^3 + 2
\delta(b)b + b\delta(b) + \delta^2(b) \neq a,$$ for all $b \in D$.
\end{theorem}

%%%%%%%%%%%%%%%%%%%%%%%%%%%%%%%%%%%%%%%%%%%%%%%%%%%%%%%%%%%%%%%%%%
%
%Irreducibility criteria
%
%%%%%%%%%%%%%%%%%%%%%%%%%%%%%%%%%%%%%%%%%%%%%%%%%%%%%%%%%%%%%%%%%%%%%%%%%

\section{Irreducibility criteria for some polynomials of degree two and three in $R=K[y][t;\sigma]$} \label{subsec:irred}

Let $K$ be a field of characteristic not $2$. Let  $R = K[y][t;\sigma]$, where $y$ is an indeterminate and
$\sigma$ an automorphism of the integral domain $K[y]$. We know that $\sigma|_{K} =id$ and $\sigma(y) = \alpha y + \beta$ for some $\alpha, \beta
\in K, \alpha \neq 0$. This implies that
 $$\sigma^2(y)=\alpha^2y+(\alpha\beta+\beta),\dots,\sigma^l(y)=\alpha^ly+\sum_{i=0}^{l}\alpha^{l-1}\beta.$$
Let $z=z(y)=\sum_{i=0}^{n}z_iy^i\in K[y]$,  then
$$\sigma(z)=\sum_{i=0}^{n}z_i(\alpha x+\beta)^i=z_n \alpha^n y^n+\dots+\sum_{j=0}^{n}z_j \beta^j$$
 has constant term
$z(\beta)=\sum_{j=0}^{n}z_j \beta^j$ and
$$\sigma^2(z)=\sum_{i=0}^{n}a_i(a^2y+\alpha\beta+\beta)^i=a_n \alpha^{2n}
y^{n}+\dots+\sum_{j=0}^{n}a_j (\alpha \beta+\beta)^j$$ has constant term
$z(\alpha \beta+\beta)=\sum_{j=0}^{n}a_j (\alpha \beta+\beta)^j=\sum_{j=0}^{n}a_j (\alpha+1)^j\beta^j$. Continuing in this manner,
\begin{equation} \label{eqn:general}
\sigma^l(z)=\sum_{i=0}^{n}a_i(\alpha^lx+\sum_{j=0}^{l}\alpha^{l-1}\beta^j)^i=a_n
\alpha^{ln} y^{n}+\dots+ \sum_{j=0}^{n}a_j (\sum_{i=0}^{l}\alpha^{l-1}\beta)^j
\end{equation}
has constant term $z(\sum_{i=0}^{l}\alpha^{l-1}\beta)=\sum_{j=0}^{n}a_j
(\sum_{i=0}^{l}\alpha^{l-1}\beta)^j$ for every integer $l\geq1$. In particular,
$$\sigma(z)z=(a_n \alpha^n y^n+\dots+\sum_{j=0}^{n}a_j \beta^j)(a_n y^n+\dots
a_0)=a_n^2 \alpha^n y^{2n}+\dots+a_0z(\beta),$$
$$\sigma^2(z)\sigma(z)z=(a_n \alpha^{2n} x^{n}+\dots+\sum_{j=0}^{n}a_j (\alpha\beta+\beta)^j)(a_n^2 \alpha^n y^{2n}+\dots+a_0z(\beta))$$
$$=a_n^3 \alpha^{3n}
y^{3n}+\dots+a_0z(\beta)z(\alpha\beta+\beta),$$
and
$$\sigma^3(z)\sigma^2(z)\sigma(z)z=(z_n \alpha^{3n}y^{n}+\dots+z(\alpha^2\beta+\alpha\beta+\beta)) (z_n^3 \alpha^{3n}
y^{3n}+\dots+z_0z(\beta)z(\alpha \beta+\beta))$$
$$=z_n^4 \alpha^{6n} y^{4n}+\dots+z_0z(\beta)z(\alpha \beta+\beta)z
(\alpha^2\beta+\alpha \beta+\beta).$$
Note that if $\beta=0$ then the constant
term of $\sigma^l(z)$ simplifies to $a_0$ and thus the constant term
of $\sigma(z)z$ to $a_0^2$,  the constant term of
$\sigma^2(z)\sigma(z)z$ to $a_0^2$, and the one of $\sigma^3(z)\sigma^2(z)\sigma(z)z$ to $a_0^4$.

Applying Theorem \ref{thm:irreducibility criteria when S is a ring without zero divisors} this means for instance:

\begin{theorem}
Let $a=\sum_{j=0}^sd_jy^j\in K[y]$ and $f(t)=t^2-a\in K[y][t;\sigma]$. Then $f(t)$ is irreducible in $K[y][t;\sigma]$
in the following cases:
\\ (i) $a(y)\in K[y]$ has odd degree,
\\ (ii) $a(y)\in K^\times\setminus K^{\times2}$,
\\ (iii) $\alpha=1$, $a(y)$ has even degree with leading coefficient not a square,
\\ (iv)  $\alpha\not=1$, $a(y)$ has even degree and leading coefficient
  not of the form $c^2 \alpha^s$ for some $c\in K^\times$, and any integer $s\geq0$,
\\ (v) $\beta=0$, $a(y)$ has even degree and $d_0\not\in K^{\times 2}$.
\end{theorem}

\begin{proof}
 $f(t)=t^2-a$ is irreducible in $K[y][t;\sigma]$ if and only if $\sigma(z)z\not=a$ for all $z\in K[y]$.
Now $\sigma(z)z=z_n^2 \alpha^n y^{2n}+\dots+z_0z(\beta)$ is either constant or a polynomial of even degree.
Thus if $a=a(y)\in K[y]$ has odd degree,
we know that $\sigma(z)z\not=a$ for all $z\in K[y]$ which shows (i).
\\
If $a\in K^\times$, we know that $\sigma(z)z=a$ is only possible for $z=a_0\in
K$ in which case we have $\sigma(z)z=z_0^2$. Thus if $a$ is not a
square in $K$, $\sigma(z)z\not=a$ for all $z\in K[y]$ which proves (ii).
\\ Suppose next
that $a=\sum_{j=0}^sd_jy^j$ has even degree. Then $\sigma(z)z=a$ for
some $z\in K[y]$ is equivalent to $z_n^2 \alpha^n y^{2n}+\dots+z_0z(\beta)=a$
for some $n$, $z_i\in K$, $z_n\not=0$. This means $d_s=z_n^2 \alpha^n$ and $d_0=z_0z(\beta)$.

If $d=1$ then for all $a$ of even degree with leading coefficient not a square, $f(t)$ is irreducible,
proving (iii).
\\ If $d\not=1$ then this implies that for all $a$ of even degree $s=2n$ with leading coefficient
  not of the form $c^2 \alpha^n$ for some $c\in K^\times$, $f(t)$ is irreducible, which shows (iv).

Moreover,  $\sigma(z)z=d$ for some $z\in K[y]$ also means $d_0=a_0z(\beta)$. Hence  if
$b=0$ and $d_0\not\in K^{\times 2}$ then $f(t)$ is irreducible. This is (v).
\end{proof}

\begin{theorem}
Let $a=\sum_{j=0}^sd_jy^j\in K[y]$ and $f(t)=t^3-a\in K[y][t;\sigma]$. Then $f(t)$ is irreducible in $K[y][t;\sigma]$ in the
following cases:
\\ (i) $a(y)\in K[y]$ has degree not divisible by 3,
\\ (ii) $a(y)\in K^\times\setminus K^{\times 3}$,
\\ (iii) $\alpha=1$, $a(y)$ has degree divisible by 3 with leading coefficient not a cube,
\\ (iv)  $\alpha\not=1$, $a(y)$ has degree divisible by 3 and leading coefficient
  not of the form $c^3 \alpha^s$ for some $c\in K^\times$, and any integer $s\geq0$,
\\ (v)  $\beta=0$, $a(y)$ has degree divisible by 3 and $d_0\not\in K^{\times 3}$.
\end{theorem}

\begin{proof}
 $f(t)=t^3-a$ is irreducible in $K[y][t;\sigma]$ if and only if
$a\not=\sigma^2(z)\sigma(z)z$ for all $z\in K[y]$ by Theorem
\ref{thm:irreducibility criteria when S is a ring without zero divisors} (iii). Now
$\sigma^2(z)\sigma(z)z=z_n^3 \alpha^{3n} y^{3n}+\dots+z_0z(\beta)z(\alpha\beta+\beta)$ is either constant or a polynomial of degree divisible by 3.
 Thus if $a=a(y)\in K[y]$ has degree not divisible by 3, we know that
$\sigma^2(z)\sigma(z)z\not=a$ for all $z\in K[y]$ proving (i).
\\ If $a\in K^\times$, we know that $\sigma^2(z)\sigma(z)z=a$
is only possible for $z=z_0\in K$ in which case we have $\sigma(z)z=z_0^3$. Thus if $a\in K^\times\setminus K^{\times 3}$,
$\sigma^2(z)\sigma(z)z\not=a$
for all $z\in K[y]$ which shows (ii).
\\ Suppose next that $a=\sum_{j=0}^sd_jy^j$ has
degree divisible by 3. Then $\sigma^2(z)\sigma(z)z=a$ for some $z\in
K[y]$ is equivalent to  $z_n^3 \alpha^{3n} y^{3n}+\dots+z_0z(\beta)z(\alpha\beta+\beta)=a$
 for some $n$, $z_i\in K$, $z_n\not=0$. This means $d_s=z_n^3 \alpha^{3n}$.

If $\alpha=1$ then this implies that for all $a$ of degree divisible by 3 with leading coefficient not a cube, $f(t)$
is irreducible  and we have proved (iii).
\\
 If $\alpha\not=1$ then this implies that for all $a$ of degree $s=3n$ with leading coefficient
not of the form $c^3 \alpha^n$ for some $c\in K^\times$, $f(t)$ is irreducible and we got (iv).
\\ Moreover, since $\sigma^2(z)\sigma(z)z=a$ for some $z\in K[y]$ also means $d_0=z_0z(\beta)z(\alpha\beta+\beta)$, we know
that if $\beta=0$ and $d_0\not\in K^{\times 3}$ then $f(t)$ is irreducible. This shows (v).
\end{proof}

Similar results can be obtained for higher degree skew polynomials, the calculations become more tedious but follow the same
pattern.

%%%%%%%%%%%%%%%%%%%%%%%%%%%%%%%%%%%%%%%%%%%%%%%%%%%%%%%%%%%%%%%%%%%%%%%%%%%%
%
% Irreducible polynomials of low degree in a quantum plane,...
%
%%%%%%%%%%%%%%%%%%%%%%%%%%%%%%%%%%%%%%%%%%%%%%%%%%%%%%%%%%%%%%%%%%%%%%%%%%%%

\section{Irreducible polynomials of low degree in a quantum plane, a quantized Weyl algebra, and in $A_h$}
\label{sec:Weyletc}

Let $K$ be a field of characteristic not $2$.
Let $R = K[y][t;\sigma,\delta]$ where $y$ is an indeterminate,
$\sigma$ an automorphism of the domain $K[y]$, i.e. $\sigma|_K=id_K$, $\sigma(y) = \alpha y + \beta$ for some $\alpha, \beta
\in K, \alpha \neq 0$, and $\delta$  a left
$\sigma$-derivation. Then $R$ is isomorphic to  a quantum plane, a quantized Weyl algebra, or  the infinite-dimensional
unital associative algebra $A_h=K[y][t;{\rm id}_{K[y]},\delta]$ studied in  \cite{BI, BII, BIII}, where
$\delta: K[y] \rightarrow K[y]$ is the $K$-linear
derivation $\delta(r) = r'h$ for some $h \in K[y]$ and $r'$ denotes
the usual derivative of $r\in K[y]$ with respect to $y$.
%Given $a(y) \in K[y]$, we denote  the degree of $a(y)\in K[y]$  by $\mathrm{deg}_y(a(y))$.

\subsection{Irreducible polynomials of low degree and $f(t)=t^m-a$ in the quantum plane}

 $\sigma$ be the automorphism of $K[y]$ such that $\sigma = id_K$ on
$K$ and $\sigma(y) = qy$ for some $1\not=q \in K^{\times}$.  Then $R=K[y][t;
\sigma]$ is a \emph{quantum plane}.

\begin{lemma} \label{quantum plane sigma(b(y))=b(qy)}
$\sigma^j(b(y)) = b(q^j y)$ for all $j \in \mathbb{N}$ and all $b(y)\in K[y]$.
\end{lemma}

\begin{proof}
If $b(y) = b_0 + b_1y + \ldots + b_l y^l \in K[y]$, then
\begin{align*}
\sigma^j(b(y)) &= \sigma^j(b_0) + \sigma^j(b_1 y) + \ldots +
\sigma^j(b_l y^l) = b_0 + b_1 \sigma^j(y) + \ldots + b_l
\sigma^j(y^l) \\ &= b_0 + b_1 q^j y + \ldots + b_l (q^j y)^l = b(q^j
y)
\end{align*}
as in Equation \eqref{eqn:general}.
\end{proof}

Lemma \ref{quantum plane sigma(b(y))=b(qy)} and Theorem
\ref{thm:irreducibility criteria when S is a ring without zero
divisors} immediately yield:

\begin{proposition} \label{Irreducibility_criteria_quantum_plane}
 (i) $f(t) = t^2 - a_1(y) t - a_0(y) \in K[y][t;\sigma]$ is irreducible if and only if
$$b(qy)b(y) - a_1(y)b(y) - a_0(y) \neq 0$$ for all $b(y) \in K[y]$.
\\ (ii) $f(t) = t^3 - a_2(y) t^2 - a_1(y) t - a_0(y) \in K[y][t;\sigma]$ is irreducible if and only if
$$b(q^2y)b(qy)b(y) - b(q^2 y)b(qy)a_2(y) - b(q^2y) a_1(qy) -
a_0(q^2y) \neq 0$$ and $$b(q^2y)b(qy)b(y) - b(qy)b(y) a_2(y) - b(y)
a_1(y) - a_0(y) \neq 0$$ for all $b(y) \in K[y]$.
\end{proposition}

\begin{corollary}
 (i) Let $f(t) = t^2 - a(y) \in K[y][t;\sigma]$ be such that ${\rm deg}(a(y))$ is odd. Then $f(t)$ is irreducible.
\\ (ii) Let $f(t) = t^3 - a(y) \in K[y][t;\sigma]$ be such that $a(y)$ is not constant and $3 \nmid {\rm deg}(a(y))$. Then $f(t)$ is irreducible.
\end{corollary}

\begin{proof}
 Let $b(y) \in K[y]$.
\\ (i) We know that ${\rm deg} \big( b(qy) b(y) \big)$ is even or 0 and so $b(qy) b(y) \neq a(y)$. Therefore $f(t)$ is irreducible by Proposition \ref{Irreducibility_criteria_quantum_plane}.
\\ (ii)  We know that ${\rm deg} \big( b(q^2y) b(qy) b(y) \big)$ is a multiple of $3$ or 0 and so $b(q^2y) b(qy) b(y) \neq a(q^2y)$ and $b(q^2y) b(qy) b(y) \neq a(y)$. Hence $f(t)$ is irreducible by Proposition \ref{Irreducibility_criteria_quantum_plane}.
\end{proof}

\begin{proposition}  \label{prop:I}
Let $f(t) = t^2 - a_1(y)t - a_0(y) \in K[y][t;\sigma]$ and let $a_{i,0}$ denote the constant terms of
$a_i(y)$, $0\leq i\leq 1$.
\\ (i) If  ${\rm deg}(a_1(y)) > {\rm
deg}(a_0(y))$, then $f(t)$ is irreducible.
\\ (ii) If $ t^2-a_{1,0}t-a_{0,0}\in K[t]$ is irreducible, then $f(t)$ is irreducible.
\end{proposition}

\begin{proof}
$f(t)$ is irreducible if and only if $b(qy)b(y) - a_1(y)b(y) \neq
a_0(y)$ for all $b(y) \in K[y]$ by Proposition
\ref{Irreducibility_criteria_quantum_plane}.
\\ (i) Suppose that
$a_0(y) = b(qy)b(y) - a_1(y)b(y) = (b(qy) - a_1(y)) b(y)$
for some $b(y) \in K[y]$ and notice that $b(y) \neq 0$ since $a_0(y)
\neq 0$. Then ${\rm deg}(b(qy) - a_1(y))$ and $ {\rm deg}(b(y)) =
{\rm deg}(b(qy)) \leq {\rm deg}(a_0(y)).$ Hence ${\rm deg}(a_1(y)) \leq {\rm deg}(a_0(y))$.
\\ (ii) A look at the constant terms of the above equation yields the assertion.
\end{proof}

\begin{proposition} \label{prop:II}
Let  $f(t)=t^3 - a_2(y) t^2 - a_1(y) t - a_0(y) \in K[y][t;\sigma]$ and let $a_{i,0}$ denote the constant terms of
$a_i(y)$, $0\leq i\leq 2$.
 If $$ t^3-a_{2,0}t^2-a_{1,0}t-a_{0,0}\in K[t]$$
  is irreducible, then
 $f(t) = t^3 - a_2(y) t^2 - a_1(y) t - a_0(y) \in K[y][t;\sigma]$
  is irreducible.
\end{proposition}

\begin{proof}
 Comparing constants in the two equations of Proposition \ref{Irreducibility_criteria_quantum_plane} (ii)
immediately yields the assertion.
\end{proof}

Note that $K[t] \subset K[y][t;\sigma]$ %since $K \subset {\rm Fix}(\sigma)$
and so if $f(t) \in K[t]$ is irreducible in $K[y][t;\sigma]$ then $f(t)$ is irreducible in $K[t]$.
Thus Propositions \ref{prop:I} and \ref{prop:II} yield:

\begin{corollary} \label{thm:irredquantum}
Let $f \in K[t] \subset K[y][t;\sigma]$.
\\ (i) $f(t) = t^2 - a_1 t - a_0$ is irreducible in $K[y][t;\sigma]$ if and only if it is irreducible in $K[t]$
 if and only if $a_1^2 + 4 a_0$ is not a square in $K$.
\\ (ii) $f(t) = t^3 - a_2 t^2 - a_1 t - a_0$ is irreducible in $K[y][t;\sigma]$ if and
 it is irreducible in $K[t]$.
\end{corollary}

The following result yields a large class of irreducible polynomials of degree $m$:

\begin{theorem}
Let $m$ be prime and $K$ contain a primitive $m$th root of unity.
Let $f(t) = t^m - a(y) \in K[y][t;\sigma]$, $ a(y)=\sum_{j=0}^sa_jy^j \not=0$.
\\ (i) If $m \nmid {\rm deg} (a(y))$, then $f(t)$ is irreducible in $K[y][t;\sigma]$.
\\ (ii) If $a_0\not\in  K^{\times m},$  then $f(t)$ is irreducible in $K[y][t;\sigma]$.
\\ (iii) If $a_s\not\in K^{\times m}q^e$ for every integer $e\geq 0,$ then $f(t)$ is irreducible in $K[y][t;\sigma]$.
\end{theorem}

\begin{proof}
 Extend $\sigma$
to an automorphism
$$\sigma(\frac{c}{d})=\frac{\sigma(c)}{\sigma(d)}$$
of the field of fractions $K(y)$ of $K[y]$ as in
\eqref{eqn:extend sigma delta to right ring of fractions}, for all $c,d\in K[y]$, $d\not=0$.
By Theorem \ref{thm:Petit(19)}, $f$ is irreducible  over
$K(y)[t;\sigma]$ and hence over $K[y][t;\sigma]$, if
$$N_m(b(y)) = \sigma^{m-1}(b(y)) \cdots \sigma(b(y)) b(y) \neq a(y)$$
for all $b(y) \in K(y)$. Write $b(y) =
c(y)/d(y)$ for some $c(y), d(y) \in K[y]$ with $d(y) \neq 0$, then
\begin{align*}
N_m(b(y)) &= \frac{c(q^{m-1}y) \cdots c(qy) c(y)}{d(q^{m-1}y) \cdots
d(qy) d(y)}
\end{align*}
by Lemma \ref{quantum plane sigma(b(y))=b(qy)}. If $N_m(b(y)) \notin
K[y]$, we immediately conclude $N_m(b(y)) \neq a(y)$ because $a(y)
\in K[y]$. So suppose that $N_m(b(y)) \in K[y]$  and  $N_m(b(y)) = a(y)$ for some $b(y) =
c(y)/d(y)$ with $c(y)=\sum_{j=0}^nc_jy^j \not=0$, $d(y)=\sum_{j=0}^ld_jy^j \not=0$.
\\ (i) If $N_m(b(y)) \in K[y]$, then
$\mathrm{deg}(N_m(b(y))) = m \mathrm{deg}(c(y)) - m
\mathrm{deg}(d(y))$ is a multiple of $m$ for all $b(y) \in K(y)$ with $\mathrm{deg} c \not=\mathrm{deg} d$. Thus $N_m(b(y))
\neq a(y)$ for all $b(y) \in K(y)$ such that $\mathrm{deg} c \not=\mathrm{deg} d$, and so $f$ is irreducible.
\\ (ii) Comparing constants in the equation $N_m(b(y)) = a(y)$ yields
$c_0^m=a_0d_0^m$, hence $a_0\in K^{\times m}.$ Thus if $a_0\not\in  K^{\times m},$
$f(t)$ is irreducible.
\\ (iii) Comparing highest terms in the equation $N_m(b(y)) = a(y)$ implies that $a_s\in  K^{\times m}q^e$
for some integer $e\geq 0$.
\end{proof}

\subsection{Irreducible polynomials of degree two in the quantized Weyl algebra}

Let $K$ be a field of characteristic not 2, $y$ be an indeterminate
and $\sigma$ be the automorphism of $K[y]$ such that $\sigma = id$ on
$K$ and $\sigma(y) = qy$ for $q \in K^{\times}, q \neq 1$. Define  $$\delta(g) = \frac{g(qy)
- g(y)}{qy - y}$$ for all $g \in K[y]$. The algebra $R=K[y][t; \sigma,
\delta]$ is  a \emph{quantized Weyl algebra}.

\begin{proposition} \label{Irreducibility_criteria_quantized_Weyl_Algebra}
 $f(t) = t^2 - a_1(y) t - a_0(y) \in K[y][t;\sigma,\delta]$ is
irreducible if and only if $$b(qy)b(y) + \frac{b(qy) - b(y)}{qy - y}
- a_1(y) b(y) - a_0(y) \neq 0$$ for all $b(y) \in K[y]$.
\end{proposition}

\begin{proof}
 By Lemma \ref{quantum plane
sigma(b(y))=b(qy)}, we have $\sigma(b(y)) = b(qy)$. Theorem \ref{prop:irreducibility criteria
when S is a ring without zero divisors delta not 0} then yields the
result.
\end{proof}

\begin{corollary} \label{thm:irredWeyl}
Let $f(t) = t^2 - a_1(y)t - a_0(y) \in K[y][t;\sigma,\delta]$ where
$a_0(y) \neq 0$.
\\ (i)
 Suppose $f(t) \in K[t]$. Then $f(t)$ is irreducible in $K[y][t;\sigma,\delta]$ if and only if $a_1^2 + 4 a_0\not\in K^\times$  if and only if $f(t)$ is irreducible in $K[t]$.
\\ (ii) Suppose $a_0(y), a_1(y)$ are such that $2 {\rm deg}(a_1(y)) < {\rm deg}(a_0(y))$ and
${\rm deg}(a_0(y))$ is odd. Then $f(t)$ is irreducible in $K[y][t;\sigma,\delta]$. In particular, if $a_1 \in K$ and
${\rm deg}(a_0(y))$ is odd then $f(t)$ is irreducible.
\end{corollary}

\begin{proof}
By Proposition
\ref{Irreducibility_criteria_quantized_Weyl_Algebra}, $f(t)$ is
irreducible in $K[y][t;\sigma,\delta]$ if and only if $$b(qy)b(y) +
\frac{b(qy) - b(y)}{qy - y} - a_1(y)b(y) \neq a_0(y)$$ for all $b(y)
\in K[y]$. this is equivalent to
\begin{equation} \label{eqn:quantised weyl degree1}
b(qy)b(y)(qy - y) + b(qy) - b(y) - a_1(y)b(y)
\neq a_0(y) (qy - y)
\end{equation}
for all $b(y) \in K[y]$.
Note that \eqref{eqn:quantised weyl degree1} is trivially
satisfied if $b(y) = 0$. If $b(y) = b \in K^{\times}$ then \eqref{eqn:quantised
weyl degree1} simplifies to
\begin{equation} \label{eqn:quantised weyl degree2}
b^2 - a_1(y)b \neq a_0(y).
\end{equation}
 (i) Suppose $a_0, a_1 \in K$. If $l={\rm deg}(b)  > 0$ then
$${\rm deg} \Big( b(qy)b(y)(qy - y) + b(qy) - b(y) -
a_1(y)b(y)(qy - y) \Big) = (2l + 1) > 1$$ and so
$$b(qy)b(y)(qy - y) + b(qy) - b(y) - a_1b(y)(qy - y)\neq a_0 (qy - y)$$
for all $b(y) \in K[y]$. Therefore $f(t)$ is irreducible in
$K[y][t;\sigma,\delta]$ if and only if \eqref{eqn:quantised weyl
degree2} is satisfied for all $b \in K$,  which is equivalent to $f(t)$ being irreducible in
$K[t]$, and this holds in turn if and only if $a_1^2 + 4
a_0$ is not a square in $K$.
\\ (ii) Suppose now that $2 {\rm deg}(a_1(y)) < {\rm deg}(a_0(y))$ and ${\rm deg}(a_0(y))$ is odd. Then for all
$b \in K^{\times}$ we have
$${\rm deg}(b^2 - a_1(y)b) = {\rm deg}(a_1(y)) < {\rm deg}(a_0(y))$$ unless $b = a_1(y) \in K$ in which case $b^2 -
a_1(y)b \in K$. In either case \eqref{eqn:quantised weyl degree2} is
satisfied.

Put $l={\rm deg}(b(y))  > 0$ then we have to consider 3 cases:
\\ If $l = {\rm deg}(a_1(y))$ then
$${\rm deg}(b(qy)b(y)(qy-y)) = 2l+1 = {\rm deg}(a_1(y)b(y)(qy-y))$$ which implies
\begin{equation*}
\begin{split}
{\rm deg} \Big( & \frac{b(qy)b(y)(qy - y) + b(qy) - b(y) -
a_1(y)b(y)(qy - y)}{qy - y} \Big) \\ & \qquad \leq (2l + 1)-1 = 2
{\rm deg}(a_1(y)) < {\rm deg}(a_0(y)).
\end{split}
\end{equation*}
\\ If $l < {\rm deg}(a_1(y))$ then
$${\rm deg}(a_1(y)b(y)(qy-y)) > {\rm deg}(b(qy)b(y)(qy-y)), \
{\rm deg}(b(qy)), \ {\rm deg}(b(y))$$ which implies
\begin{equation*}
\begin{split}
{\rm deg} \Big( & \frac{b(qy)b(y)(qy - y) + b(qy) - b(y) -
a_1(y)b(y)(qy - y)}{qy - y} \Big) \\ &= {\rm deg}(a_1(y)) + l + 1
-1 \\&=  {\rm deg}(a_1(y)) + l < 2 {\rm deg}(a_1(y)) < {\rm deg}(a_0(y)).
\end{split}
\end{equation*}
\\ Finally, if $l > {\rm deg}(a_1(y))$ then
$${\rm deg}(b(qy)b(y)(qy-y)) > {\rm deg}(a_1(y)b(y)(qy-y)), \
{\rm deg}(b(qy)), \ {\rm deg}(b(y))$$ which implies
\begin{equation*}
\begin{split}
{\rm deg} \Big( & \frac{b(qy)b(y)(qy - y) + b(qy) - b(y) -
a_1(y)b(y)(qy - y)}{qy - y} \Big) = (2l + 1)-1
\end{split}
\end{equation*}
is even.

In all cases we have $$b(qy)b(y)(qy - y) + b(qy) - b(y)-a_1(y)b(y)(qy - y)\neq a_0(y)(qy - y),$$
therefore $f(t)$ is irreducible.
\end{proof}

\subsection{Irreducible polynomials of degree two in $A_h$}

Recall that $A_h=K[y][t; \delta]$ with $\delta(r) = r'h$ for some $h(y)
\in K[y]$, where $r'$ is the usual derivation of $r$ with respect to $y$.
The algebras $A_h$ were comprehensively studied in \cite{BI, BII, BIII}. $A_h$ is simple if and only if $F$ has characteristic 0
and $h\in F^\times$ \cite[Corollary 7.5]{BI}.
 If $F$ has characteristic 0 then $A_h$ is a unique factorization domain \cite[Lemma 7.6]{BI}.

The irreducibility of a polynomial in $A_h$ clearly depends on the choice of $h$:

\begin{proposition} \label{Irreducibility_criteria_Ah_degree_2}
 $f(t) = t^2 - a_1 (y)t - a_0(y) \in K[y][t; \delta]$ is
irreducible if and only if
\begin{equation*}
b(y)^2 + b'(y)h(y) - a_1(y) b(y) - a_0(y) \neq 0
\end{equation*}
for all $b(y) \in K[y]$.
\end{proposition}

\begin{proof}
We have $\sigma(b(y))b(y) + \delta(b(y)) - a_1(y)b(y) - a_0(y) =
b(y)^2 + b'(y)h(y) - a_1(y) b(y) - a_0(y).$ The assertion follows
from Theorem \ref{prop:irreducibility criteria when S is a ring
without zero divisors delta not 0}.
\end{proof}

\begin{corollary} \label{A_h degree irreducible t^2-a_1(y)t-a_0(y)}
Let $f(t) = t^2 - a_1(y)t - a_0(y) \in K[y][t;\delta]$ and let $a_{i,0}$ denote the constant terms of
$a_i(y)$, $0\leq i\leq 1$.
\\ (i) If
$${\rm deg}(a_0(y)) > 2 {\rm max} \{ {\rm deg}(a_1(y)), {\rm deg}(h(y))-1 \}$$ and ${\rm deg}(a_0(y))$ is odd, then $f(t)$ is
irreducible in $K[y][t;\delta]$.
\\ (ii) If $h(y)$ has zero constant term and $g(t)=t^2-a_{1,0}t-a_{0,0}$ is irreducible in $K[t]$,
 then $f(t)$ is irreducible in $K[y][t;\delta]$.
\end{corollary}

\begin{proof}
$f(t)$ is irreducible in $K[y][t;\delta]$ if and only if
\begin{equation} \label{eqn:A_h degree irreducible t^2-a_1(y)t-a_0(y)}
b(y)^2 + b'(y)h(y) - a_1(y)b(y) \neq a_0(y)
\end{equation}
for all $b(y) \in K[y]$ by Proposition
\ref{Irreducibility_criteria_Ah_degree_2}.
\\ (i) If $b(y) = 0$
then \eqref{eqn:A_h degree irreducible t^2-a_1(y)t-a_0(y)} is
satisfied since $a_0(y) \neq 0$. If $b(y) \in K^{\times}$ then
$b'(y)h(y)=0$ which implies $${\rm deg} \big( b(y)^2 + b'(y)h(y) -
a_1(y)b(y) \big) = {\rm deg}(a_1(y)) < {\rm deg}(a_0(y)),$$
unless $a_1(y) \in K$ in which case $b(y)^2 + b'(y)h(y) - a_1(y)b(y)
\in K$. In either case \eqref{eqn:A_h degree irreducible
t^2-a_1(y)t-a_0(y)} is satisfied.

Now suppose $l={\rm deg}(b(y)) > 0$. We  consider the following two cases:
\\ If $l > {\rm max} \{ {\rm deg}(a_1(y)), {\rm deg}(h(y))-1 \}$ then
\begin{equation*}
\begin{split}
{\rm deg} \big( b(y)^2 + & b'(y)h(y) - a_1(y)b(y) \big) = {\rm deg}(b(y)^2) = 2l \\ &> 2 {\rm max} \{ {\rm deg}(a_1(y)), {\rm deg}(h(y))-1 \}
\end{split}
\end{equation*}
and is even.
\\ If $l \leq {\rm max} \{ {\rm deg}(a_1(y)), {\rm deg}(h(y))-1 \}$ then
\begin{align*}
{\rm deg} \big( b(y)^2 + b'(y)h(y) - a_1(y)b(y) \big) & \leq {\rm
max} \{ {\rm deg}(a_1(y)) + l, {\rm deg}(h(y)) - 1 + l \} \\ &= l
+ {\rm max} \{ {\rm deg}(a_1(y)), {\rm deg}(h(y)) - 1 \} \\ &
\leq 2 {\rm max} \{ {\rm deg}(a_1(y)), {\rm deg}(h(y)) - 1 \}.
\end{align*}
Therefore if ${\rm deg}(a_0(y)) > 2 {\rm max} \{ {\rm deg}(a_1(y)), {\rm deg}(h(y))-1 \}$ and ${\rm deg}(a_0(y))$ is
odd then \eqref{eqn:A_h degree irreducible t^2-a_1(y)t-a_0(y)} is
satisfied which implies $f(t)$ is irreducible.
\\ (ii) Suppose there exists some $b(y)$ such that $b(y)^2 + b'(y)h(y) - a_1(y)b(y) = a_0(y)$. Let $b(y)=\sum_{i=0}^sb_iy^i$.
Looking at the constant terms the equation then yields $b_0^2 - a_{1,0}b_0 = a_{0,0}$.
Thus if  $g(t)=t^2-a_{1,0}t-a_{0,0}$ is irreducible in $K[t]$ there is no such $b(y)$ and the assertion follows.
\end{proof}

\begin{corollary} \label{A_h t^2-a degree irreducible}
Let $f(t) = t^2 - a(y) \in K[y][t;\delta]$ and $h(y) \in K[y]$.
\\ (i) Suppose ${\rm deg} (h(y))\in\{0,1\}$. If ${\rm deg}(a(y))$ is odd then $f(t)$ is irreducible.
\\ (ii) Suppose ${\rm deg}(h(y)) = n \geq 2$ and ${\rm deg}(a(y)) \geq 2n - 2$ is odd. Then $f(t)$ is irreducible.
\end{corollary}

\begin{proof}
Set $a_1(y) = 0$ in Corollary \ref{A_h degree irreducible t^2-a_1(y)t-a_0(y)} (i).
\end{proof}

%*******************************************************************************************%
%****************************************************************************************%


\begin{thebibliography}{1}


\bibitem{BI} G.~Benkart, S. A.~Lopes, M.~Ondrus,
 \emph{A parametric family of subalgebras of the Weyl algebra I.
  Structure and automorphisms.} Trans. Amer. Math. Soc. 367 (3) (2015), 1993-2021.

\bibitem{BII} G.~Benkart, S. A.~Lopes, M.~Ondrus,
 \emph{A parametric family of subalgebras of the Weyl algebra II. Irreducible modules.}
 Online at     arXiv:1212.1404 [math.RT]

\bibitem{BIII} G.~Benkart, S. A.~Lopes, M.~Ondrus,
 \emph{Derivations of a parametric family of subalgebras of the Weyl algebra}. J. Algebra 424 (2015), 46-97.

\bibitem{B} N. Boubaki, \emph{\'El\'ements de math\'ematique. Fasc. XXIII. Alg\`ebre; Chapitre 8,
modules e anneaux semi-simples.} Hermann, 1973.

\bibitem{CP18} C. Brown, S. Pumpl\"un, \emph{How a nonassociative algebra reflects the properties of a skew polynomial.}
 Glasgow Mathematical Journal (Nov. 2019)
\\  https://doi.org/10.1017/S0017089519000478

\bibitem{C} P. M. Cohn, ``Skew fields. Theory of general division rings.'' Encyclopedia of Mathematics and its Applications.
Cambridge University Press, 1995.

\bibitem{CookP} C.~Cook, K.~Price, \emph{Factorization in quantized Weyl algebras.} Preprint,
University of Wisconsin Oshkosh
August 24, 2005

\bibitem{CP} R.~Coulibaly, K.~Price, \emph{Factorization in quantum planes.} Missouri J. Math. Sci. 18
(3)(2006), 197-205.

\bibitem{GLN} J. G\`{o}mez-Torrecillas, F. J. Lobillo, G. Navarro, \emph{
Computing the bound of an Ore polynomial. Applications to factorization.}
J. Symbolic Comp. 2018, online at  https://doi.org/10.1016/j.jsc.2018.04.018


\bibitem{Goodearl} K. Goodearl, \emph{Prime ideals in skew polynomial rings and quantized Weyl algebras.}
J. Algebra 150 (2) (1992), 324-377.


\bibitem{HP} C.~Holtz, K.~Price, \emph{Normal quadratics in Ore extensions, quantum planes, and
quantized Weyl algebras.} Acta Applicandae Mathematicae 108 (1) (2009), 73-81.


\bibitem{J96} N.~Jacobson, ``Finite-dimensional division algebras over fields.'' Springer
Verlag, Berlin-Heidelberg-New York, 1996.

\bibitem{lam1988vandermonde} T. Y. Lam,   A.~Leroy, \emph{Vandermonde and Wroskian matrices over division rings.}
J. Alg. 19(2)(1988), 308-336.


\bibitem{P66} J.-C. Petit, \emph{Sur certains quasi-corps g\'{e}n\'{e}ralisant un type d'anneau-quotient}.
S\'{e}minaire Dubriel. Alg\`{e}bre et th\'{e}orie des nombres 20
(1966 - 67), 1-18.

\bibitem{O1} O. Ore, \emph{Theory of noncommutative polynomials.} Annals of Math.  34 (3) (1933), 480-508.


\end{thebibliography}
\end{document}